\documentclass[a4paper,10pt,twoside]{article}

\usepackage[pdftex]{graphicx}
\DeclareGraphicsExtensions{.pdf,.PDF,.eps,.EPS,.png,.PNG,.jpg,.JPG,.jpeg,.JPEG}
\usepackage{lastpage}
\usepackage{latexsym}
\usepackage{dsfont}
\usepackage{amsmath,amsfonts,amssymb,amsthm}
\usepackage{geometry}
\usepackage{bera} 
\usepackage[pdftex,pagebackref=false]{hyperref}
\hypersetup{pdfborder=0 0 0}
\hypersetup{pdfstartview={FitH}}

\usepackage[utf8]{inputenc}

\newtheorem{theorem}{Theorem}[section]
\newtheorem{proposition}[theorem]{Proposition}%
\newtheorem{corollary}[theorem]{Corollary}%
\newtheorem{lemma}[theorem]{Lemma}%

\newcommand{\ACKNO}[1]{\noindent\textbf{Acknowledgments.} #1}

\makeatletter
\def\erf{\mathop{\operator@font erf}\nolimits}
\def\erfc{\mathop{\operator@font erfc}\nolimits}
\def\argmax{\mathop{\operator@font argmax}\nolimits}
\def\prob{\mathop{\operator@font Prob}\nolimits}
\@ifundefined{widetext}{}{}
\makeatother

\begin{document}

\title{\Large Small random instances of the stable roommates problem}

\author{Stephan Mertens\\[2ex]
  {\small Otto-von-Guericke Universität Magdeburg, Germany and Sante Fe
  Institute, USA} \\
  {\small mertens@ovgu.de}}

\date{\small February 20, 2015}

\maketitle

\begin{abstract}
  Let $p_n$ denote the probability that a random instance of the
  stable roommates problem of size $n$ admits a solution. We derive an
  explicit formula for $p_n$ and compute exact values of $p_n$ for
  $n\leq 12$.
\end{abstract}

\section{Introduction}
\label{sec:intro}

Matching under preferences is a topic of great practical importance,
deep mathematical structure, and elegant algorithmics
\cite{manlove:book,gusfield:irving:book}. A paradigmatic example is
the stable roommates problem \cite{gale:shapley:62}. Consider an even
number $n$ of participants. Each of the participants ranks all the
others in strict order of preference. A matching is a set of $n/2$
disjoint pairs of participants. A matching is stable if there is no
pair of unmatched participants who both prefer each other to their
partner in the matching. Such a pair is said to block the matching.
The stable roommates problem is to find a stable matching. The name
originates from the problem to assign students to the double bedroomes
of a dormitory. Another application is the formation of cockpit crews
from a pool of pilots.

An instance of the stable roommates problem is defined by a preference
table, in which each participant ranks all other $n-1$ participants,
most preferred first. For technical reasons we will assume that each
participant puts himself at the very end of his preference list. Here
are two examples for $n=4$:
\begin{equation}
  \label{eq:examples}
  \text{(A)}\quad
    \begin{array} {rrrrr}
    1: & 4 & \mathbf{2} & 3 & 1\\
    2: & 3 & 4 & \mathbf{1} & 2\\
    3: & 1 & \mathbf{4} & 2 & 3\\
    4: & \mathbf{3} & 2 & 1 & 4
  \end{array}
  \qquad\qquad
  \text{(B)}\quad
  \begin{array} {rrrrr}
    1: & 3 & 2 & 4 & 1\\
    2: & 1 & 3 & 4 & 2\\
    3: & 2 & 1 & 4 & 3\\
    4: & 1 & 2 & 3 & 4
  \end{array}
\end{equation}
In (A), the marked matching $(1,2)(3,4)$ is stable. In (B), there is
no stable matching: whoever is matched with $4$ can always form a
blocking pair with someone else.
Example (B) illustrates the fact that not all instances of the
stable roommates problem have a solution. Let $p_n$ denote the
probabilty that a random instance, chosen uniformely from all possible
instances of size $n$, admits a solution.  Our examples shows that $0
< p_4 <1$. The exact value is $p_4=26/27$. It has been computed by
Pittel \cite{pittel:93a} more than 20 years ago. No other values of
$p_n$ are known exactly. Numerical simulations
\cite{mertens:roommates} suggest that $p_n$ is a monotonically
decreasing function of $n$ that asymptotically decays like $n^{-1/4}$.

In this paper we derive an explicit formula for $p_n$ that we use
to compute exact values of $p_n$ for $n\leq 12$. And we 
discuss a generalization of this approach for odd
values of $n$.

\section{Stable Permutations}
\label{sec:stable-permutations}

A matching of size $n$ can be interpreted as a permutation $\pi$ of
$\{1,\ldots,n\}$ that is completely composed of 2-cyles.  An obvious
generalization is to allow arbitrary permutations $\pi$, but for that
one needs to extend the definition of stability.  A permutation $\pi$
is called stable if it satisfies the two following conditions:
\begin{subequations}
\begin{equation}
  \label{eq:stable-partition-1}
  \mbox{$\forall i: i$ does not prefer $\pi(i)$ to $\pi^{-1}(i)$}  
\end{equation}
\begin{equation}
  \label{eq:stable-partition-2}
  \mbox{$i$ prefers $j$ to $\pi(i) \Rightarrow j $ prefers $\pi(j)$ to $i$}
\end{equation}
\end{subequations}
This definition includes permutations with fixed points. This is the
reason why we've added each participant to the very end of his own
preference list. But note that \eqref{eq:stable-partition-2} rules out
that a stable permutation can have more than one fixed point.

For permutations composed of $2$-cycles (matchings) condition
\eqref{eq:stable-partition-1} is trivially satisfied and condition
\eqref{eq:stable-partition-2} reduces to the usual ``no blocking
pairs'' condition. Condition \eqref{eq:stable-partition-1} enforces
each cycle of length $\geq 3$ to have a monotonic rank ordering: every
member $i$ prefers his predecessor $\pi^{-1}(i)$ to his successor
$\pi(i)$, and condition \eqref{eq:stable-partition-2} prevents any
member of the cycle to leave the cycle.

The significance of stable permutations for the stable roommates
problem arises from the following facts, proven by Tan \cite{tan:91}:
\begin{enumerate}
\item Each instance of the stable roommates problem admits at least
  one stable permutation.
\item \label{fact2} If $\pi$ is a stable permutation for a roommates
  instance that contains a cycle $C=(v_1,v_2,\ldots,v_{2m})$ of even
  length, we can get two different stable permutations by replacing $C$ by the $2$-cycles
  $(v_1,v_2),\ldots,(v_{2m-1},v_{2m})$ or by
  $(v_2,v_3),\ldots,(v_{2m},v_{1})$.
\item If $C$ is an odd-length cycle in \emph{one} stable permutation
  for a given roommates instance, then $C$ is a cycle in \emph{all}
  stable permutations for that instance.
\end{enumerate}
These facts establish the \emph{cycle type} of stable permutations as
certificate for the existence of a stable matching:  An instance of
the stable roommates problem is solvable if and only if the instance
admits a stable permutation with no odd cycles.

Consider again the two examples from the previous section.  One can
easily check that the permutation $(1,2,3)(4)$ is a stable permutation
for $(B)$. Since it contains the odd cycle $(1,2,3)$, (B) admits no
stable matching. The permutation $(1,3,4,2)$ is stable for
(A). According to fact \ref{fact2}, its $4$-cycle can be replaced by
$(1,3)\,(4,2)$ or by $(3,4)\,(1,2)$, which are in fact both stable
matchings.

\section{A Formula for $\mathbf{p_n}$}
\label{sec:theory}

The facts proven by Tan allow us to derive an explicit formula for the
probability $p_n$. The underlying ideas have already been
discussed more or less in \cite{pittel:93a}, but the formulas
\eqref{eq:Pn-general} and \eqref{eq:complement-Pn-general} haven't
been published before. We start with an integral representation for
$P(\pi)$, the probability that a permutation $\pi$ is stable.
\begin{proposition} \label{the:PnPi}
  Let $\pi$ be a permutation of $\{1,\ldots,n\}$ and let $F_\pi =
  \{i:i=\pi(i)\}$ denote the fixed points and $M_\pi = \{i:
  \pi(i)=\pi^{-1}(i)\neq i\}$ the elements in two cycles of $\pi$. The
  probability that $\pi$ is a stable permutation for a random instance
  of the stable roommates problem is given by
\begin{equation}
  \label{eq:PnPi}
  P(\pi) = \int_0^1 \mathrm{d}^nx \, \prod_{(i, j> i)\not\in D_\pi} (1-x_jx_i)
  \prod_{i\not\in M_\pi \cup F_\pi} x_i \prod_{i\in F_\pi}\delta(x_i-1)\,,
\end{equation}
where integration is over the $n$-dimensional unit cube and
\begin{equation}
  \label{eq:def-Dpi}
  D_\pi = \{(i,j) : i\neq j\,,\,\,i = \pi(j) \vee j = \pi(i) \}
\end{equation}
is the set of pairs of elements that are cyclic neighbors in $\pi$.
\end{proposition}
\begin{proof}
A random instance of the stable roommates problem can be generated as
follows: Introduce an $n\times (n-1)$ array of independent random variable
$X_{ij}$ $(1 \leq i\neq j\leq n)$, each uniformly distributed in
$[0,1]$. Each agent $i$ ranks the agents $j\neq i$ on his preference
list in increasing order of the variables $X_{ij}$. Obviously, such an
ordering is uniform for every $i$, and the orderings by different
members are independent. The fact that each agent is at the very end
of his hown preference list is taken into account by adding variables
$X_{ii}=1$ to the set of random variables.

Let $P(\pi|x,y)$ denote the conditional probability that the
permutation $\pi$ is stable given $X_{i\pi(i)}=x_i$ and
$X_{i\pi^{-1}(i)}=y_i$, and  let $F_\pi = \{i:i=\pi(i)\}$ and
$M_\pi = \{i: \pi(i)=\pi^{-1}(i)\neq i\}$ denote the fixed
points and two cycles of $\pi$. Then \eqref{eq:stable-partition-1}
 tells us
\begin{equation}
  \label{eq:integrand-1}
  P(\pi|x,y) \propto \prod_{i\not\in\mathcal{M}_\pi\cup\mathcal{F}_\pi}
  \Theta(x_i-y_i)\,\prod_{i \in M_\pi\cup
    F_\pi}\delta(x_i-y_i)
\end{equation}
where $\Theta$ is the step function
\begin{displaymath}
  \Theta(z) = \begin{cases}
     1 & z \geq 0 \\
     0 & z < 0
  \end{cases}
\end{displaymath}
and $\delta(z)$ is the Dirac delta function.

The second condition \eqref{eq:stable-partition-2} is violated if $X_{ij} < x_i$ and
$X_{ji} < x_j$ for some $(i,j)\not\in D_\pi$. This happens with
probability $x_ix_j$, hence
\begin{equation}
  \label{eq:integrand-2}
  P(\pi|x,y) \propto \prod_{(i, j>i)\not\in D_\pi} (1-x_jx_i)\,,
\end{equation}
which does not depend on $y$. 

Integrating \eqref{eq:integrand-1} over $y_i$ gives a factor
$x_i$ if $i$ is an element of cycle of length three or more, a factor
$1$ otherwise.  Adding the product $\prod_{i\in F_\pi} \delta(x_i-1)$
to ensure the constraints $X_{ii}=1$ finally allows us to integrate
over the $x_i$'s to obtain \eqref{eq:PnPi}.
\end{proof}

Note that \eqref{eq:PnPi} differs slightly from the integral
representation in \cite{pittel:93a}: Our integral is valid for any
permutation $\pi$. If $\pi$ contains more
than one fixed point, the integrand vanishes since the
$\delta$-function forces at least one of the factors in the product
$\prod (1-x_ix_j)$ to be zero and $P(\pi)=0$ as it should.

Obviously $P(\pi)$ depends on $\pi$ only through the cycle type of
$\pi$. Let $a_k$ denote the number of cycles of length $k$ in $\pi$.
We use the notation $\mathbf{a}=[1^{a_1}, 2^{a_2},\ldots]$ to
denote the cycle type, including only those terms with
$a_k>0$. For $n=4$, the only
non-zero integrals are
\begin{subequations}
\label{eq:p4-integrals}
\begin{equation}
  \label{eq:Pa-4-02}
  P([2^2]) = \int_0^1 \mathrm{d}^4x \left( 1-x_{{1}}x_{{3}} \right)  \left( 1-x_{{1}}x_{{4}} \right) 
 \left( 1-x_{{2}}x_{{3}} \right)  \left( 1-x_{{2}}x_{{4}} \right) = \frac{233}{648} 
\end{equation}
\begin{equation}
  \label{eq:Pa-4-0001}
  P([4^1]) = \int_0^1 \mathrm{d}^4x \left( 1-x_{{1}}x_{{3}} \right)  \left( 1-x_{{2}}x_{{4}} \right) x_{{1}}x_{{2}}x_{{3}}x_{{4}} = \frac{25}{1296}
\end{equation}
\begin{equation}
  \label{eq:Pa-4-101}
  P([1^1\,3^1]) = \int_0^1 \mathrm{d}^3x\,(1-x_1)(1-x_2)(1-x_3)x_1x_2x_3 = \frac{1}{216}\,.
\end{equation}
\end{subequations}
Note that in the last integral, we have already done the trivial integration over $\delta(x_4-1)$.

\begin{proposition}
  \label{the:Pn-general}
  Let $p_n$ ($n$ even) be the probability that a random instance of the stable roommates
  problem has a solution. Then 
  \begin{equation}
    \label{eq:Pn-general}
    p_n = \sum_{\mathbf{a}\in\mathcal{E}_n} (-1)^{e(\mathbf{a})} c(\mathbf{a}) P(\mathbf{a})\,,
  \end{equation}
  where $\mathcal{E}_n$ is the set of all cycle types of size $n$ with
  even cycles only.  The exponent $e(\mathbf{a})$ is the number of
  even cycles of length $\geq 4$ in $\mathbf{a}$,
  $e(\mathbf{a})=\sum_{k=4,6,\ldots} a_k$. The factor
  $c(\mathbf{a})$ is the number of permutations with cycle type
  $\mathbf{a}$,
  \begin{equation}
    \label{eq:num-cycles}
    c(\mathbf{a}) = \frac{n!}{\prod_k a_k!\, k^{a_k}}\,.
  \end{equation}
\end{proposition}
\begin{proof}
  A matching of size $n$ has cycle structure $\mathbf{a}=[2^{n/2}]$,
  and there are $(n-1)!!$ matchings of size $n$. Boole's inequality
  (aka union bound) then tells us that
  \begin{equation}
    \label{eq:boole-1}
    p_n \leq (n-1)!!\, P([2^{n/2}])\,,
  \end{equation}
  where equality holds if and only if the stability of different matchings 
  were independent. This is not true in our case. Fact \ref{fact2}
  from above tells us that stable matchings may come in pairs. Every stable
  permutation that consists of exactly one even length cycle of size $z\geq 4$ and
  $(n-z)/2$ cycles of size $2$ corresponds to two stable matchings. These pairs
  have been counted twice in the sum in \eqref{eq:boole-1}.
  The number of permutations of cycle type $[2^{(n-z)/2}\,z^1]$ is 
  $n!\,\big((n-z)!!\,z\big)^{-1}$ and we get
  \begin{equation}
    \label{eq:boole-2}
    p_n \geq (n-1)!!\, P([2^{n/2}]) - \sum_{z=4,6,\ldots}^n \frac{n!}{(n-z)!!\,z} P([2^{(n-z)/2}\,z^1])\,.
  \end{equation}
  The $\geq$ is again a consequence of Boole's inequality. Equality in
  \eqref{eq:boole-2} would only hold if the stability of pairs of
  permutations were independent events, but we know from fact
  \ref{fact2} that stable pairs again may come in pairs: we have a quartet of
  stable permutation for each permutation that is composed of precisely
  two cycles of length $\geq 4$ and $2$-cycles. Again we can express
  the corrections by $P([])$ and a combinatorial prefactor. Iterating this reasoning
  (which is of course the well known inclusion-exclusion principle) yields
  \eqref{eq:Pn-general}.

  The formula \eqref{eq:num-cycles} for the number of permutations
  of a given cycle type is well known. Yet we will give a short proof
  for completeness. Write down the cycle structure
  in terms of $a_k$ pairs of parentheses enclosing $k$ dots, like
  \begin{equation}
    \label{eq:cycle-structure}
    (\cdot \cdot \cdot) (\cdot \cdot \cdot) (\cdot \cdot) (\cdot)
  \end{equation}
  for $n=9$ and $\mathbf{a}=[1^1,\,2^1,\,3^2]$. Now imagine that the $n$ dots
  are replaced left to right with a permutation of
  $\{1,\ldots,n\}$. Then the parentheses induces the desired cycle
  structure on this permutation. There are $n!$ permutations, but some
  of them result in the same "cycled" permutations. First, a cycle of
  length $k$ can have $k$ different leftmost values in $(\cdots)$, which
  gives a factor $k^{a_k}$ of overcounting. And pairs of parentheses
  that hold the same number of dots can be arranged in any order, which
  gives a factor $a_k!$ of overcounting. This yields \eqref{eq:num-cycles}.
\end{proof}

\begin{corollary}
  \label{the:complement-Pn-general}
  Let $\mathcal{O}_n$ denote the set of all cycle types of size $n$ that contain
  at most one fixed point and at least one odd cycle. Then
  \begin{equation}
     \label{eq:complement-Pn-general}
     1-p_n = \sum_{\mathbf{a}\in\mathcal{O}_n} (-1)^{e(\mathbf{a})} c(\mathbf{a}) P(\mathbf{a})\,,
  \end{equation}
\end{corollary}
\begin{proof}
   Since $P(\mathbf{a})=0$ if $\mathbf{a}$ has more than one fixed
   point, we can extend the sum to run over all cycle types with at
   least one odd cycle. Then the right hand side of
   \eqref{eq:complement-Pn-general} is the probability that a random 
   instance of the stable roommates problem has a stable permutation 
   with at least one odd cycle. But this equals the probability that a
   random instance of the stable roommates problem has no solution.
\end{proof}


\section{Evaluation of $p_n$}

We already know the values of the integrals $P(\mathbf{a})$ for $n=4$,
see \eqref{eq:p4-integrals}.
When we insert these values into \eqref{eq:Pn-general} or \eqref{eq:complement-Pn-general} we get
\begin{equation}
  \label{eq:P4}
  \begin{aligned}
  p_4 &= 3\, P([2^2]) - 6\, P([4^1]) = {\frac {26}{27}} =
  0.962962\ldots\,\\
  1-p_4 &= 8\, P([1^1,3^1]) = \frac{1}{27}\,,
  \end{aligned}
\end{equation}
the value computed by Pittel in 1993 \cite{pittel:93a}. It seems
straightforward to compute $p_n$ for larger values of $n$, since all
we need to do is to evaluate and sum the corresponding integrals $P(\mathbf{a})$.
This is not easy, however.  Pittel
wrote ``For $n = 6$, the computations by hand become
considerably lengthier and we gave up after a couple of half-hearted attempts.'' 

The computations become ``lengthier'' for two reasons: the number of
integrals in \eqref{eq:Pn-general} and
\eqref{eq:complement-Pn-general} increase with $n$, and the evaluation
of each individual integral gets harder. Let us first look at the
number of integrals:
\begin{lemma}
  Let $p(n)$ denote the number of unordered partitions of $n$, and let
  $n$ be even. Then
  \begin{subequations}
  \begin{align}
    |\mathcal{E}_n| &= p\left(\frac{n}{2}\right) \label{eq:En}\\
    |\mathcal{O}_n| &= p(n)-p(n-2)-p\left(\frac{n}{2}\right)\label{eq:On}\,.
  \end{align}
  \end{subequations}
\end{lemma}
\begin{proof}
  From $\sum_k k a_k=n$, or from glancing at
  \eqref{eq:cycle-structure}, it is obvious that there is a one-to-one
  correspondence between the set of cycle types of size $n$ and the
  set of integer partitions of $n$. 

  Every cycle type $\mathbf{a}\in\mathcal{E}_n$ corresponds to a
  partition of $n$ into even numbers and vice versa. Every partition
  of $n$ into even numbers corresponds to a unique partition of $n/2$
  and vice versa---simply divide or mutiply all parts of the
  partition by two. This proves \eqref{eq:En}.

  The number of all cycle types is $p(n)$, and the number of all cycle
  types that contain at least two fixed points is $p(n-2)$. Hence the
  number of cycle types that contain at most one fixed point is
  $p(n)-p(n-2)$. For $|\mathcal{O}_n|$ we also need to subtract the number
  of cycle types with even cycles only, which is $p(n/2)$. This proves \eqref{eq:On}.
\end{proof}

There is no closed formula for the partition numbers $p(n)$, but they
are known for all $n \leq 10\,000$ \cite{oeisA000041}. And
we need $p(n)$ only for small values of $n$ to get
\begin{center}
\begin{tabular}{crrrrrrrr}
  $n$ & 4 & 6 & 8 & 10 & 12 & 14 & 16 & 18 \\[1ex]
  $|\mathcal{E}_n|$ & 2 & 3 & 5 & 7 & 11 & 15 & 22 & 30 \\
  $|\mathcal{O}_n|$ & 1 & 3 & 6 & 13 & 24 & 43 & 74 & 124 
\end{tabular}
\end{center}
In this regime of $n$, the number of integrals is no problem. So let
us turn our attention to the individual integrals.

When we expand the product in \eqref{eq:PnPi}, we get a sum of
easy-to-integrate terms of the form $x_1^{b_1}\cdot  x_n^{b_n}$,
but there too many terms to be integrated by hand. 
\begin{lemma}
A full expanion of the integrand in \eqref{eq:PnPi} yields
$2^{f(\mathbf{a})}$ terms, where
\begin{equation}
  \label{eq:factors}
  f(\mathbf{a}) = \frac{1}{2} n (n-3) + a_1 + a_2\,.
\end{equation}
\end{lemma}
\begin{proof}
If we expand the integrand, each factor in the product
\begin{equation}
  \label{eq:culprit}
   \prod_{\substack{i<j \\ (i,j)\not\in D_\pi}}^n (1-x_i x_j)
\end{equation}
doubles the number of terms. Hence we need to show that
\eqref{eq:factors} is the number of factors in this product.
Think of the $n$ variables $x_i$ as the vertices of a graph $G$.
Each factor $(1-x_ix_j)$ in \eqref{eq:culprit} corresponds to an edge
of $G$. Without the constraint $(i,j)\not\in D_\pi$, $G$ is the
complete graph with $\frac{1}{2}n(n-1)$ edges. Each cycle of length
$k\geq 3$ in $\mathbf{a}$ corresponds to a cycle in $G$ with $k$ edges
that are removed from the complete graph. Each cycle of length $2$
corresponds to an edge that is also removed. This gets us
\begin{displaymath}
  f(\mathbf{a}) = \frac{1}{2}n(n-1)-\sum_{k\geq 3} k a_k - a_2 =
  \frac{1}{2}n(n-1)- \left(\sum_{k} k a_k - 2a_2 -a_1\right) -a_2
\end{displaymath}
and \eqref{eq:factors} follows from $\sum_k k a_k = n$.
\end{proof} 

The maximum number of terms arises for pure matchings, i.e., for
$a_2=n/2$ and $a_1=0$. It reads $2^4$, $2^{12}$, $2^{24}$, $2^{40}$
and $2^{60}$ for $n=4,6,8,10$ and $12$. Hence it is no surprise that
Pittel gave up on the integrals for $n=6$. The integration is better
left to a computer.


We used the computer-algebra system Mathematica \cite{mathematica:10}
for the exact evaluation of the integrals
$P(\mathbf{a})$. Figure~\ref{fig:integration} shows the Mathematica
code that sets up the integrand and performs the integration. The full
Mathematica code is available online \cite{mertens:page:matchings}.

\begin{figure}
  \centering\small
\begin{verbatim}
Integrand[a_] := Module[(* computes integrand corresponding to cycle pattern a *)
  {n,i,j,l,result,cycle},
  If[a[[1]]>1,result=0, (* more than one fixed point *)
    n=Sum[k*a[[k]],{k,1,Length[a]}];
    If[a[[1]]>0, (* take care of fixed point *)
      n=n-1;result=Product[(1-x[i]),{i,1,n}],
      result = 1 (* no fixed point *)
    ];
    result=result*Product[(1-x[i]*x[j]),{i,1,n-1},{j,i+1,n}];
    (* remove 2-cycles from product *)
    result=result/Product[(1-x[2*i-1]*x[2*i]),{i,1,a[[2]]}];
    (* cycles larger than 2 *)
    result=result*Product[x[i],{i,2*a[[2]]+1,n}];
    For[cycle=3,cycle<=Length[a],cycle++,
      l=Sum[i*a[[i]],{i,2,cycle-1}]+1;
      For[i=l,i<=l+cycle*(a[[cycle]]-1),i+=cycle,
        For[j=0,j<cycle,j++,result=result/(1-x[i+j]*x[i+Mod[(j+1),cycle]])]
      ]
    ]
  ];
  result
];

P[a_] := Module[
  {y,n,k},
  n=Sum[k*a[[k]],{k,1,Length[a]}];
  If[a[[1]]>0,n=n-1];
  y = Integrand[a];
  For[k=n,k>=1,k--,y=Integrate[y,{x[k],0,1}]];
  y
];
\end{verbatim}
  \caption{Mathematica code to compute the integrals $P(\mathbf{a})$
    \eqref{eq:PnPi}. The procedure \texttt{Integrand[a]} returns the
    integrand as a function of variables
    \texttt{x[1]},\ldots,\texttt{x[n]} (or \texttt{x[n-1]} if the
    cycle type \texttt{a} contains a fixed point), the procedure
    \texttt{P[a]} evaluates the integral by exactly integrating
    variable by variable.}
  \label{fig:integration}
\end{figure}

Using our Mathematica code, we computed the values of $p_n$ for $n\leq 12$ both from
\eqref{eq:Pn-general} and (as a crosscheck) from \eqref{eq:complement-Pn-general}. 
The results are
\begin{subequations}
  \label{eq:Pvalues}
\begin{align}
  \label{eq:P6}
  p_6 &= {\frac {181431847}{194400000}} = 0.93329139403292181070\ldots\\[1ex]
  \label{eq:P8}
  p_8 &= {\frac {809419574956627}{889426440000000}} = 0.91004667564933981499\ldots\\[1ex]
  \label{eq:P10}
  p_{10} &= \frac{25365465754520943457921774207}{28460490127321448448000000000}
      = 0.89125189485653484085\ldots\\[1ex]
  \label{eq:P12}
  p_{12} &=
  \frac{13544124829485098788469430650439043569062157071}{15469783933925839494793980316271247360000000000}
  \\
   &= 0.87552126696367780620\ldots\,.\nonumber
\end{align}
\end{subequations}
The values of the corresponding individual integrals are listed in
Tables~\ref{tab:n10} and \ref{tab:n12}.

\begin{table}
  \centering
  \begin{tabular}{rrrrrr}
    & \multicolumn{1}{c}{$p_4$} & \multicolumn{1}{c}{$p_6$} & \multicolumn{1}{c}{$p_8$} & \multicolumn{1}{c}{$p_{10}$} & \multicolumn{1}{c}{$p_{12}$} \\[1ex]
\eqref{eq:Pn-general}: & 0.20 sec. & 19.8 sec. & 5 min. & 20 min. & 15.5
days \\[1ex]
\eqref{eq:complement-Pn-general}: & 0.02 sec. & 3.5 sec. & 6 min. & 25
min. &  13.9 days
  \end{tabular}
  \caption{\label{tab:times}Times to compute $p_n$ according
    \eqref{eq:Pn-general} or \eqref{eq:complement-Pn-general}.}
\end{table} 

We ran our Mathematica code on a computer equipped with 2
Intel\textsuperscript{\textregistered}\
Xeon\textsuperscript{\textregistered}\ CPUs E5-1620 with 3.60 GHz
clock rate and 32 GByte of memory. The total computation times are
shown in Table~\ref{tab:times}. Table~\ref{tab:n12} also shows the
times to compute the individual integrals for $n=12$. Some of theses
integrals (marked with a $\star$) could not be computed by the simple
iterative scheme in Figure~\ref{fig:integration} because Mathematica
ran out of memory. In these cases we expanded the integrand in a
polynomial in the variable $x_n$ (or $x_{n-1}$ if there is a fixed
point) and applied interative integration to each coefficient of this
polynomial. This reduces the memory consumption, but it slows down the
computation. With a larger memory (like 64 GByte instead of 32 GByte),
this could have been avoided and $p_{12}$ could
have been computed somewhat faster.

\begin{table}
  \centering
  \begin{tabular}{c@{\hskip 0mm}c@{\hskip 8mm}c@{\hskip 0mm}c}
    $\mathbf{a}$ & $P(\mathbf{a})$ & $\mathbf{a}$ & $P(\mathbf{a})$  \\[1ex]
    $[2^3]$ & $\frac{448035973}{5832000000}$ & $[1^1,2^1,3^1]$ & $\frac{38077}{86400000}$ \\[1ex]
    $[2^1,4^1]$ & $\frac{307841}{144000000}$ & $[1^1,5^1]$ & $\frac{26257}{777600000}$\\[1ex]
    $[6^1]$ & $\frac{2591729}{11664000000}$ & $[3^2]$ & $\frac{1742111}{7776000000}$\\[3ex]
    $[2^4]$ &$\frac {1245959394495647}{107585022182400000}$ & $[1^1,7^1]$ & $\frac {49958102093}{384232222080000000}$  \\[1ex]
    $[2^2,4^1]$ & $\frac {5211637894488503}{26896255545600000000}$ & $[1^1,2^2,3^1]$ & $\frac {441974732789}{12807740736000000}$\\[1ex]
    $[2^1,6^1]$ & $\frac {914248620325799}{53792511091200000000}$ & $[1^1,3^1,4^1]$ & $\frac {1249592153}{9605805552000000}$ \\[1ex]
    $[4^2]$ & $\frac {1493807915753}{1195389135360000000}$ &  $[1^1,2^1,5^1]$ & $\frac {58105985423}{25615481472000000}$ \\[1ex]
    $[8^1]$ & $\frac {622186155317}{498078806400000000}$ & $[2^1,3^2]$ & $\frac {76670733315619}{4482709257600000000}$ \\[1ex]
    & & $[3^1,5^1]$ & $\frac {58105985423}{25615481472000000}$ \\[3ex]
  $[2^5]$ &
  $\frac{433857166916418660757431885203}{322741958043825225400320000000000}$
  & $[1^1,2^3,3^1]$ & $\frac{1882697003227025150390719}{819662115666857715302400000000}$\\[1ex]
  $[2^3,4^1]$ &
  $\frac{4794693488032751578104859937}{322741958043825225400320000000000}$
  & $[1^1,3^3]$ & $\frac{158398327239405983477}{512288822291786072064000000000}$ \\[1ex]
  $[2^2,6^1]$ &
  $\frac{726158117631681830112186713}{645483916087650450800640000000000}$
  & $[1^1,2^1,3^1,4^1]$ & $\frac{2765878679393466620633}{409831057833428857651200000000}$ \\[1ex]
  $[2^1,4^2,]$ &
  $\frac{94089601969271248978571831}{1290967832175300901601280000000000}$
  & $[1^1,2^2,5^1]$ & $\frac{4336602947669955694769}{32786484626674308612096000000}$ \\[1ex]
    $[2^1,8^1]$ & 
    $\frac{18812621042800384360939621}{258193566435060180320256000000000}$
  & $[1^1,4^1,5^1]$ & $\frac{126601947989502609349}{409831057833428857651200000000}$ \\[1ex]
    $[4^1,6^1]$ &
    $\frac{10678226865621944175135083}{2581935664350601803202560000000000}$
  & $[1^1,3^1,6^1]$ & $\frac{633196266619396193087}{2049155289167144288256000000000}$ \\[1ex]
    $[10^1]$ &
    $\frac{42708804188035567140443357}{10327742657402407212810240000000000}$
  & $[1^1,2^1,7^1]$ &
  $\frac{789921304062168675601}{117094587952408245043200000000}$
  \\[1ex]
 & & $[1^1,9^1]$ &
 $\frac{1265995491264426770353}{4098310578334288576512000000000}$\\[1ex]
 & & $[2^2,3^2]$ &
 $\frac{2918990176269285877130918549}{2581935664350601803202560000000000}$\\[1ex]
& & $[2^1,3^1,5^1]$ &
$\frac{18845369089082632479619357}{258193566435060180320256000000000}$\\[1ex]
& & $[3^2,4^1]$ &
$\frac{21410287713579117222366871}{5163871328701203606405120000000000}$
\\[1ex]
& & $[3^1,7^1]$ &
$\frac{610894828667022260751797}{147539180820034388754432000000000}$
\\[1ex]
& & $[5^2]$ & $\frac{8541874436295301342281403}{2065548531480481442562048000000000}$
  \end{tabular}
  \caption{\label{tab:n10}Probabilities $P(\mathbf{a})$ for $n=6,8,10$
    (top to bottom). Cycle
  types with (right) and without (left) odd cycles.}
\end{table}

\begin{table}
  \centering
  \begin{tabular}{ccr}
    $\mathbf{a}$ & $P(\mathbf{a})$ & Time [sec.]  \\[1ex]
      $[2^6]$ & $\frac{325899908494883644126440199857602193757211429627}{2572934463890545624774134806202233860915200000000000}$ & 265\,018\\[1ex]
$[2^4,4^1]$ & $\frac{1209115974791734652605681563324122140963407221}{1225206887566926487987683241048682790912000000000000}$ & 265\,091\makebox[0pt]{\:\:$^\star$}\\[1ex]
$[2^2,4^2]$ & $\frac{16288072152327610053000950409164225186650151}{4288224106484242707956891343670389768192000000000000}$  & 205\,089\makebox[0pt]{\:\:$^\star$}\\[1ex]
$[4^3]$ & $\frac{231703173390597300042186053017177445722753}{25729344638905456247741348062022338609152000000000000}$  & 206\,115\makebox[0pt]{\:\:$^\star$}\\[1ex]
$[2^3,6^1]$ & $\frac{3378294177941932509053172872298924486923764591}{51458689277810912495482696124044677218304000000000000}$ & 49\,493\\[1ex]
$[2^1,4^1,6^1]$ & $\frac{417880592074077264470531487240272595070729}{2144112053242121353978445671835194884096000000000000}$ & 49\,220\\[1ex]
$[6^2]$ & $\frac{1853330912748299530044034784734880537462353}{205834757111243649981930784496178708873216000000000000}$ & 49\,293\\[1ex]
$[2^2,8^1]$ & $\frac{508893194633666952579907861671385135829521}{134007003327632584623652854489699680256000000000000}$ & 47\,303\\[1ex]
$[4^1,8^1]$ & $\frac{4412925241742005785167715449536219676971}{490082755026770595195073296419473116364800000000000}$ & 47\,520\\[1ex]
$[2^1,10^1]$ & $\frac{53484730261191253361608747405394814514357}{274446342814991533309241045994904945164288000000000}$ & 53\,836\\[1ex]
$[12^1]$ &
$\frac{14826641669894164340076941832557808383893}{1646678056889949199855446275969429670985728000000000}$
& 53\,299 \\[3ex]
   $[1^1,2^4,3^1]$ &
   $\frac{122503966894472107602242737308438169403}{920820472257490446118689304539955200000000000}$
   & 92\,605\\[1ex]
  $[1^1,2^1,3^3]$ &
  $\frac{1222543880169122622560877738575059}{93543667022983156431104945223106560000000000}$
  & 42\,996 \\[1ex]
  $[1^1,2^2,3^1,4^1]$  &
  $\frac{193959334006722457965074605079586629657}{618791357357033579791759212650849894400000000000}$
  & 7\,038\\[1ex]
   $[1^1,3^1,4^2]$  &
   $\frac{2855025200767172513735081732106863}{5729549605157718331405177894915276800000000000}$
   & 8\,354\\[1ex] 
   $[1^1,2^3,5^1]$ &
   $\frac{8437895290055710585350317910566101247317}{1237582714714067159583518425301699788800000000000}$
   & 10\,784\\[1ex]
   $[1^1,3^2,5^1]$ &
   $\frac{176230039945164631684723423501203937}{353595061346876331309576692943342796800000000000}$ & 10\,498\\[1ex]
   $[1^1,2^1,4^1,5^1]$ &
   $\frac{1795748861495201189078302039999290739}{137509190523785239953724269477966643200000000000}$ & 9\,092\\[1ex]
   $[1^1,2^1,3^1,6^1]$ &
   $\frac{2694291584800385097759305707384840499}{206263785785677859930586404216949964800000000000}$ & 8\,314\\[1ex]
   $[1^1,5^1,6^1]$ &
   $\frac{2174740904996317920876887889334183}{4365371127739213966784897443744972800000000000}$ & 8\,376\\[1ex]
   $[1^1,2^2,7^1]$ &
   $\frac{1695852720842076492466118915028628181}{5412169306329739764942500402194022400000000000}$ & 8\,059\\[1ex]
   $[1^1,4^1,7^1]$ &
   $\frac{88077935438211707375963113857429259}{176797530673438165654788346471671398400000000000}$ & 8\,082\\[1ex]
   $[1^1,3^1,8^1]$  &
   $\frac{19270992061340957619880212582236803}{38674459834814598736984950790678118400000000000}$ & 8\,098\\[1ex]
   $[1^1,2^1,9^1]$ &
   $\frac{1197147888812853403164542654655246797}{91672793682523493302482846318644428800000000000}$ & 8\,096\\[1ex]
   $[1^1,11^1]$  &
   $\frac{3699232196202777873480674898554033731}{7425496288284402957501110551810198732800000000000}$ & 8\,001\\[1ex]
   $[2^3,3^2]$ & $\frac{161499154693883709213457621140881273357670713}{2450413775133852975975366482097365581824000000000000}$ & 209\,450\\[1ex]
   $[2^2,3^1,5^1]$ &
   $\frac{4348643545825433788892694617856351275828603}{1143526428395798055455171024978770604851200000000000}$ & 56\,207\\[1ex]
   $[2^1,3^2,4^1]$ & $\frac{357327697339220191285941924523831564051}{1829642285433276888728273639966032967761920000000}$ & 212\,763\makebox[0pt]{\:\:$^\star$}\\[1ex]
   $[3^4]$ & $\frac{9836824308843120655187019024812769030613}{1089072788948379100433496214265495814144000000000000}$ & 216\,063\\[1ex]
   $[2^1,3^1,7^1]$ &
   $\frac{91054335946285045516721350625722458874631}{466745480977876757328641234685212491776000000000000}$ & 50\,230\\[1ex]
   $[3^2,6^1]$  &
   $\frac{618748213165565813756023362735829302961763}{68611585703747883327310261498726236291072000000000000}$ & 49\,614\\[1ex]
  $[2^1,5^2]$ &
  $\frac{26742627021755978677974945561880197844453}{137223171407495766654620522997452472582144000000000}$ & 59\,920\\[1ex]
  $[3^1,4^1,5^1]$ &
  $\frac{61829720130534204153121031185033846815523}{6861158570374788332731026149872623629107200000000000}$ & 59\,587\\[1ex]
  $[3^1,9^1]$ &
  $\frac{20608745217756332777239346633128265563123}{2287052856791596110910342049957541209702400000000000}$ & 53\,642\\[1ex]
  $[5^1,7^1]$ &
  $\frac{17650822382212529975478949518938933533441}{1960331020107082380780293185677892465459200000000000}$
  & 53\,495
  \end{tabular}
  \caption{\label{tab:n12}Probabilities $P(\mathbf{a})$ for $n=12$
    and the times to compute them. Times marked with
    $\star$ refer to a slower, more memory efficient integration procedure (see text).}
\end{table}

\section{Odd values of $n$}

For odd values of $n$ there are no stable matchings, of
course. But there are still stable permutations: Tan's results listed 
in Section~\ref{sec:stable-permutations} as well as
Proposition~\ref{the:PnPi} also hold for odd values of $n$.
This allows us to generalize the stable roommates problem to odd
values of $n$. The most obvious generalization is to accept one
fixed point, i.e., to reject one participant from the dormitory (or
put him into a single bedroom), and to ask for a stable matching of
the remaining $n-1$ participants. Let $p_n$ (for $n$ odd) denote the
probability that a random instance admits such a solution. Following
the same reasoning as in Proposotion~\ref{the:Pn-general} and
Corollary~\ref{the:complement-Pn-general}, we get
\begin{align}
    \label{eq:Pn-odd-general}
    p_n &= \sum_{\mathbf{a}\in\mathcal{E}^1_n} (-1)^{e(\mathbf{a})}
    c(\mathbf{a}) P(\mathbf{a})\,, \\
   \label{eq:complement-Pn-odd-general}
    1-p_n &= \sum_{\mathbf{a}\in\mathcal{O}^3_n} (-1)^{e(\mathbf{a})} c(\mathbf{a}) P(\mathbf{a})\,,
\end{align}
where $\mathcal{E}^1_n$ is the set of all cycle types of size $n$
consisting of one fixed point and even cycles and $\mathcal{O}^3_n$
is the set of all cycle types of size $n$ that contain at least one
cycle of odd length $\geq 3$. Table~\ref{tab:n-odd} lists the values of the corresponding integrals
$P(\mathbf{a})$ for odd $n\leq 11$. The resulting values of $p_n$ are
\begin{subequations}
  \label{eq:Pvalues-odd}
\begin{align}
  p_3 &= \frac{3}{4} = 0.75\\[1ex]
  \label{eq:P5}
  p_5 &= \frac{4075}{6912} = 0.5895543981481481\ldots\\[1ex]
  \label{eq:P7}
  p_7 &= \frac{246462083}{518400000} = 0.4754284008487654\ldots\\[1ex]
  \label{eq:P9}
  p_{9} &= \frac{11365049284140796201}{29144725585920000000} = 0.38995218021992023\ldots\\[1ex]
  \label{eq:P11}                                                                                        
  p_{11} &= \frac{176967745750762518431538515329}{546441410444571810201600000000} = 0.3238549318705289\ldots
\end{align}
\end{subequations}
It seems counterintuitive that $p_{2k-1} < p_{2k}$, but note that the
enforced fixed-point for an odd number of participants represents
someone who is happy to be matched with anybody else.
This high destabilizing potential is a result of the rule that every participant
has to put himself at the very end of his preference list.

\section{Conclusions and Outlook}

We have seen that $p_n$, the probabilty of a random instance of the
stable roommmates problem of size $n$ to admit a solution, can be expressed as a sum over
cycle types of permutations of size $n$. Each term in the sum is an
integral with an exponential number of terms. The latter restricts an
exact evaluation of $p_n$ to $n\leq 12$. In spite of this limitation, the
method is far more efficient than the exhaustive enumeration over the
$[(n-1)!]^{n-1}$ different instances of size $n$. For
$n=12$, this number is $4.1\times 10^{83}$, or $4100$ times the number
of atoms in the visible universe (which is usually estimated as
$10^{80}$).

Our results for $n\leq 12$ don't shed new light on the ultimate
behavior of $p_n$ as $n$ becomes large, but they suggest that exact
evaluation of $p_n$ for any larger values of $n$ is likely to be
infeasible without some unexpected new approach.

The approach outlined in this paper can easily be modified to work for
the stable matching problem on general graphs, where each participant
corresponds to a vertex of a graph $G$ and ranks only those
participants adjacent to him in $G$. If $G$ is the complete graph, we
recover the stable roommates problem. In the case of bipartite graphs
$G$ (known as stable marriage problem) we have $p_n=1$. For
non-bipartite graphs, $p_n$ seems to be a monotonically decreasing
function of $n$ that may or may not approach a non-zero value,
depending on the number of short cycles in $G$
\cite{mertens:matchings}.

\begin{table}
  \centering
  \begin{tabular}{c@{\hskip 0mm}c@{\hskip 8mm}c@{\hskip 0mm}c}
    $\mathbf{a}$ & $P(\mathbf{a})$ & $\mathbf{a}$ & $P(\mathbf{a})$
    \\[1ex]
    $[1^1,2^1]$ & $\frac{1}{4}$ &
$[3^1]$ & $\frac{1}{8}$ \\[3ex]
    $[1^1,2^2]$ & $\frac{833}{20736}$ & $[2^1,3^1]$ &
    $\frac{491}{27648}$ \\[1ex]
    $[1^1,4^1]$ & $\frac{1}{2304}$ & $[5^1]$ & $\frac{191}{82944}$
    \\[3ex]
   $[1^1,2^3]$ & $\frac{110831617}{23328000000}$ & $[2^2,3^1]$
& $\frac{5103637}{2592000000}$ \\[1ex]
  $[1^1,2^1,4^1]$ & $\frac{797731}{23328000000}$ & $[2^1,5^1]$ &
$\frac{1945639}{9331200000}$
  \\[1ex]
$[1^1,6^1]$ & $\frac{6541}{2916000000}$ & $[3^1,4^1]$ & $\frac{336349}{18662400000}$\\[1ex]
$[1^1,3^2]$ & $\frac{2183}{972000000}$& $[7^1]$ & $\frac{558779}{31104000000}$\\[3ex]
$[1^1,2^4]$ & $\frac{12242957448855683129}{27541765678694400000000}$ &$[2^3,3^1]$ & $\frac{39406434169244998649}{220334125429555200000000}$
\\[1ex]
$[1^1,2^2,4^1]$ & $\frac{2998148628185909}{1311512651366400000000}$ & $[2^2,5^1]$ & $\frac{234360972607515209}{14688941695303680000000}$\\[1ex]
$[1^1,4^2]$ & $\frac{184134811312313}{27541765678694400000000}$ & $[2^1,3^1,4^1]$ & $\frac{3502136387768779}{2937788339060736000000}$\\[1ex]
$[1^1,2^1,6^1]$ & $\frac{3617070987119831}{27541765678694400000000}$ & $[3^3]$ & $\frac{374799675933251}{4896313898434560000000}$\\[1ex]
$[1^1,8^1]$ & $\frac{26303739761759}{3934537954099200000000}$ &$[2^1,7^1]$ & $\frac{12476274579169301}{10492101210931200000000}$
\\[1ex]
$[1^1,2^1,3^2]$ & $\frac{1206877128048157}{9180588559564800000000}$ & $[3^1,6^1]$ & $\frac{16811008475015879}{220334125429555200000000}$\\[1ex]
$[1^1,3^1,5^1]$ & $\frac{2455964944171}{367223542382592000000}$ & $[4^1,5^1]$ & $\frac{671436255551711}{8813365017182208000000}$\\[1ex]
& & $[9^1]$ & $\frac{1864835590786319}{24481569492172800000000}$\\[3ex]
$[1^1,2^5]$&$\frac{88853486478784120344992170351}{2581935664350601803202560000000000}$&$[2^4,3^1]$&$\frac{710107424563570828306588840739}{51638713287012036064051200000000000}$\\[1ex]
$[1^1,2^3,4^1]$&$\frac{4572509990406797552502341}{34425808858008024042700800000000}$&$[2^3,5^1]$&$\frac{109197089334060411167876570117}{103277426574024072128102400000000000}$\\[1ex]
$[1^1,2^1,4^2]$&$\frac{268054718171660435931803}{860645221450200601067520000000000}$&$[2^2,3^1,4^1]$& $\frac{14352373021321999225705658471}{206554853148048144256204800000000000}$\\[1ex]
$[1^1,2^2,6^1]$&$\frac{1168831786137020235667067}{172129044290040120213504000000000}$&$[2^1,3^3]$& $\frac{183002200715285406357445301}{45901078477344032056934400000000000}$\\[1ex]
$[1^1,4^1,6^1]$&$\frac{743715115011041403407}{57376348096680040071168000000000}$&$[2^2,7^1]$& $\frac{204627732127480795488157591}{2950783616400687775088640000000000}$\\[1ex]
$[1^1,2^1,8^1]$&$\frac{100516822545753167453891}{322741958043825225400320000000000}$&$[2^1,3^1,6^1]$& $\frac{547557978971950371021494551}{137703235432032096170803200000000000}$\\[1ex]
$[1^1,10^1]$&$\frac{66933419040890225282203}{5163871328701203606405120000000000}$&$[2^1,4^1,5^1]$& $\frac{2461017693717356460362113427}{619664559444144432768614400000000000}$\\[1ex]
$[1^1,2^2,3^2]$ & $\frac{4386643900008678909343237}{645483916087650450800640000000000}$&$[3^2,5^1]$& $\frac{41866526759821300816071211}{206554853148048144256204800000000000}$\\[1ex]
$[1^1,3^2,4^1]$&$\frac{531499853646597948383}{40983105783342885765120000000000}$&$[3^1,4^2]$& $\frac{373490614662460067378083}{1844239760250429859430400000000000}$\\[1ex]
$[1^1,2^1,3^1,5^1]$&$\frac{804392338445761377188767}{2581935664350601803202560000000000}$&$[2^1,9^1]$& $\frac{182277891278802756936253723}{45901078477344032056934400000000000}$\\[1ex]
$[1^1,5^2]$&$\frac{16733378688533122315949}{1290967832175300901601280000000000}$&$[3^1,8^1]$& $\frac{62737571161936687651226813}{309832279722072216384307200000000000}$\\[1ex]
$[1^1,3^1,7^1]$&$\frac{19128779897378689455131}{1475391808200343887544320000000000}$&$[4^1,7^1]$& $\frac{17908596195396917111551979}{88523508492020633252659200000000000}$\\[1ex]
$[11^1]$& $\frac{62675660640300931114214381}{309832279722072216384307200000000000}$&$[5^1,6^1]$& $\frac{6963996535691809265274221}{34425808858008024042700800000000000}$
  \end{tabular}
  \caption{\label{tab:n-odd}Probabilities $P(\mathbf{a})$ for $n=3,5,7,9,11$
    (top to bottom).}
\end{table}  

\bibliographystyle{unsrt} 
\bibliography{mertens,matchings,math}

\ACKNO{I am grateful for comments by David Manlove and Rob Irving.}

\end{document}